\documentclass[11pt]{article}
\usepackage[margin=1in]{geometry}
\usepackage[utf8]{inputenc}
\usepackage[T1]{fontenc}
\usepackage[english]{babel}
\usepackage{lmodern}
\usepackage{graphicx}
\usepackage{wrapfig}
\usepackage[subrefformat=parens,labelformat=parens]{subfig}
\usepackage{caption}
\usepackage{paralist}

\usepackage[linesnumbered,ruled,vlined]{algorithm2e}

\graphicspath{{./figures/}}

\makeatletter
\renewcommand{\p@enumii}{}
\makeatother

\newcommand{\titel}{Find Subtrees of Specified Weight and Cycles of Specified Length in Linear Time}
\usepackage[usenames]{color}
\definecolor{hellblau}{rgb}{0.2,0.4,1} 
\definecolor{dunkelblau}{rgb}{0,0,0.8}
\definecolor{dunkelgruen}{rgb}{0,0.5,0}
\usepackage[
	pdftex,
	colorlinks,
	linkcolor=dunkelblau,
	urlcolor=dunkelblau,
	citecolor=dunkelgruen,
	bookmarks=true,
	linktocpage=true,
	pdfauthor={On-Hei Solomon Lo},
	pdfsubject={}
]{hyperref}
\usepackage{pdfpages}
\usepackage{amsmath}
\usepackage{amsthm}
\usepackage{amsfonts}
\theoremstyle{plain}
	\newtheorem{satz}{Satz}[]
	\newtheorem{theorem}[satz]{Theorem}
	\newtheorem{lemma}[satz]{Lemma}
	\newtheorem{observation}[satz]{Observation}
	
	\newtheorem{corollary}[satz]{Corollary}
\theoremstyle{remark}
	
\theoremstyle{definition}

\makeatletter
\newcommand{\setword}[2]{%
	\phantomsection
	#1\def\@currentlabel{\unexpanded{#1}}\label{#2}%
}
\makeatother

\begin{document}
	\title{\titel}
		\author{
		On-Hei Solomon Lo\thanks{School of Mathematical Sciences, University of Science and Technology of China, Hefei, Anhui 230026, China.}
		\thanks{Institut f\"{u}r Mathematik, Technische Universit\"{a}t Ilmenau, Weimarer Strasse 25,
			D-98693 Ilmenau, Germany. This work was partially supported by the DFG grant SCHM 3186/1-1 and by DAAD (as part of BMBF, Germany) and the Ministry of Education Science, Research and Sport of the Slovak Republic within the project 57320575.}\\
		}
	\date{}
	\maketitle

\begin{abstract}
	We apply the Euler tour technique to find subtrees of specified weight as follows. Let $k, g, N_1, N_2 \in \mathbb{N}$ such that $1 \leq k \leq N_2$, $g + h > 2$ and $2k - 4g - h + 3 \leq N_2 \leq 2k + g + h - 2$, where $h := 2N_1 - N_2$. Let $T$ be a tree of $N_1$ vertices and let $c : V(T) \rightarrow \mathbb{N}$ be vertex weights such that $c(T) := \sum_{v \in V(T)} c(v) = N_2$ and $c(v) \leq k$ for all $v \in V(T)$. We prove that a subtree $S$ of $T$ of weight $k - g + 1 \leq c(S) \leq k$ exists and can be found in linear time. We apply it to show, among others, the following: \begin{itemize}
		\item Every planar hamiltonian graph $G = (V(G), E(G))$ with minimum degree $\delta \geq 4$ has a cycle of length $k$ for every $k \in \{\lfloor \frac{|V(G)|}{2} \rfloor, \dots, \lceil \frac{|V(G)|}{2} \rceil + 3\}$ with $3 \leq k \leq |V(G)|$.
		\item Every $3$-connected planar hamiltonian graph $G$ with $\delta \geq 4$ and $|V(G)| \geq 8$ even has a cycle of length $\frac{|V(G)|}{2} - 1$ or $\frac{|V(G)|}{2} - 2$.
	\end{itemize} Each of these cycles can be found in linear time if a Hamilton cycle of the graph is given. This work was partially motivated by conjectures of Bondy and Malkevitch on cycle spectra of 4-connected planar graphs. 
\end{abstract}

\section{Introduction} \label{sec:intro}
Given a tree $T$ and vertex weights $c: V(T) \rightarrow \mathbb{N}$, it is natural to ask subtrees of which specified weight would exist. Let $S$ be a subtree of $T$. We define $c(S) := \sum_{v \in V(S)} c(v)$. Let $k, g \in \mathbb{N}$ with $1 \leq k \leq c(T)$. We aim at finding a subtree $S$ of weight $k - g + 1 \leq c(S) \leq k$. 
Denote by $N_1$ and $N_2$ the number of vertices of the tree $T$ and the weight $c(T)$ of the whole tree, respectively. Note that if we allow $N_2$ to be arbitrarily large when compared to $N_1$, then it would be hopeless for us to achieve our goal. For example, it can happen that every vertex has weight, say $g' \gg g$, then a subtree $S$ of weight $k - g + 1 \leq c(S) \leq k$ exists if and only if $k \equiv 0, \dots, g - 1 \mod g'$. It means that the desired subtree does not exist for most choices of $k$. We describe by $h$ the difference between $N_1$ and $\frac{N_2}{2}$. Our main goal is to prove the following lemma, which can be interpreted as that the closer the value of our target $k$ to the medium-weight $\frac{N_2}{2}$ and the smaller the medium-weight $\frac{N_2}{2}$ when compared to the number of vertices $N_1$, the more favourable conditions we have in finding the desired subtree. It is complemented by a deterministic linear-time algorithm, which will be given in Section~\ref{subsec:proof}.

\begin{lemma} \label{lem:ksubtree}
	Let $k, g, N_1, N_2 \in \mathbb{N}$ 
	such that 
	$1 \leq k \leq N_2$, $g + h > 2$ and $2k - 4g - h + 3 \leq N_2 \leq 2k + g + h - 2$, where $h := 2N_1 - N_2$. Let $T$ be a tree of $N_1$ vertices and let $c : V(T) \rightarrow \mathbb{N}$ be vertex weights such that $c(T) = N_2$ and $c(v) \leq k$ for all $v \in V(T)$. Then there exists a subtree $S$ of $T$ of weight $k - g + 1 \leq c(S) \leq k$ and $S$ can be found in $O(N_1)$ time.
\end{lemma} 

For the running time we assume that each arithmetic operation can be done in constant time.

There are numerous results concerning subtrees of a tree with vertex weights, e.g. partitioning a tree into subtrees with constraints, finding a subtree of maximum weight. However, the author is not aware of result similar to ours on finding one subtree of specified weight.

Along with the existence of subtrees of specified weight, we present an optimal linear-time algorithm for finding them. Note that a tree may have exponentially many subtrees in general. Hence, in our algorithm only a rather restricted (linear-size) subclass of subtrees will be considered. We will exploit the Euler tour technique and find a subtree by local search. Formally, given a tree $T = (V(T), E(T))$, we construct a directed cycle $C_T$ of size $2|E(T)|$ and consider the canonical homomorphism which maps vertices of $C_T$ to that of $T$ (as depicted in Figure~\ref{fig:CT}). It is clear that every path in $C_T$ will be correspondingly mapped to a subtree of $T$. We will prove that it is indeed the case that there exists such a subtree satisfying the requirement. Therefore a linear-time algorithm follows as a simple consequence, which searches greedily for a path in $C_T$ such that its corresponding subtree in $T$ is what we are looking for. To prove that such a path in $C_T$ exists, we assume it is not the case, and then deduce a contradiction by counting the number of vertices of weight 1 in $T$ in two ways. Section~\ref{sec:main} is devoted to the proof of Lemma~\ref{lem:ksubtree}.

The problem on finding cycles of specified length appears as one of the most fundamental problems in algorithmic graph theory. By a novel method called \emph{color-coding}, Alon et al.~\cite{Alon1995} gave a randomized algorithm which finds a cycle of length $k$ in linear expected time for a fixed $k$ and a planar graph $G = (V(G), E(G))$ containing such a cycle. It can be derandomized at the price of a $\log |V(G)|$ factor. We refer to~\cite{Alon1997, Feder2010} for more related results on finding cycles of specified length efficiently. In this paper we will apply Lemma~\ref{lem:ksubtree} and obtain some closely related results. We prove, among others, the following: \begin{itemize}
	\item Every planar hamiltonian graph $G$ with minimum degree $\delta \geq 4$ has a cycle of length $k$ for every $k \in \{\lfloor \frac{|V(G)|}{2} \rfloor, \dots, \lceil \frac{|V(G)|}{2} \rceil + 3\}$ with $3 \leq k \leq |V(G)|$.
	\item Every $3$-connected planar hamiltonian graph $G$ with $\delta \geq 4$ has a cycle of length $\frac{|V(G)|}{2} - 1$ or $\frac{|V(G)|}{2} - 2$ if $|V(G)| \geq 8$ is even.
\end{itemize} Each of these cycles can be found in linear time if a Hamilton cycle of the graph is given. These results were partially motivated by two conjectures, one posed by Bondy in 1973 and another one by Malkevitch in 1988. A detailed account of this subject will be given in Section~\ref{sec:cycle}.

Finding a subtree of specified weight in a tree can be seen as a harder problem than finding a subset of specified sum in a multiset of integers, as one can weight the vertices of the tree by the integers in the multiset. We refer to Appendix~\ref{sec:number} for further discussion.

\section{Notation} \label{sec:notation}

We use minus sign to denote set subtraction, and parentheses would be omitted for single elements if it causes no ambiguity.

We consider only simple graphs in this paper. Let $G$ be an undirected graph. We denote by $V(G)$ and $E(G)$ the vertex set and the edge set of $G$ and call $|V(G)|$ and $|E(G)|$ order and size of $G$, respectively. We denote by $d_G(v)$ the degree of vertex $v \in V(G)$ in the graph $G$. The minimum and maximum degrees of $G$ are defined as $\delta(G) := \min_{v \in V(G)} d_G(v)$ and $\Delta(G) := \max_{v \in V(G)} d_G(v)$, respectively. For $W \subseteq V(G)$, $G[W]$ is defined to be the induced subgraph of $G$ on $W$. Let $c : V(G) \rightarrow \mathbb{N}$ be vertex weights. For $i \in \mathbb{N}$, we denote by $V_i(G) \subseteq V(G)$ the set of vertices $v$ in $G$ with $c(v) = i$. We write $V_i := V_i(G)$ if there is no ambiguity. Let $H$ be a subgraph of $G$, we define $c(H) := \sum_{v \in V(H)} c(v)$.

In an undirected graph $G$ we denote by $vw$ or $wv$ the edge with endvertices $v, w \in V(G)$. We abuse the notation of a sequence of vertices as follows. Let $t \in \mathbb{N}$. For $t$ distinct vertices $v_1, v_2, \dots, v_t$, we denote by $v_1 v_2 \dots v_t$ the \emph{path} $P$ with endvertices $v_1, v_t$ such that $V(P) := \{v_1, v_2, \dots, v_t\}$ and $E(P) := \{v_1 v_2, v_2 v_3, \dots, v_{t - 1} v_t\}$. For $t \geq 3$ distinct vertices $v_1, v_2, \dots, v_t$, we denote by $v_1 v_2 \dots v_t v_1$ the \emph{cycle} $K$ of length $t$ such that $V(K) := \{v_1, v_2, \dots, v_t\}$ and $E(K) := \{v_1 v_2, v_2 v_3, \dots, v_{t - 1} v_t, v_t v_1\}$.

In a directed graph $G$ we denote by $vw$ the edge directed from $v$ to $w$ for $v, w \in V(G)$. Let $C$ be a directed cycle. For $u, v \in V(C)$, we define $[u, v]_C$ to be the path directed from $u$ to $v$ along $C$.
Subscripts can be omitted if it is clear from the context. 
Let $vw$ be an edge in $C$, we define $v^+ := w$ and $w^- := v$.

For a plane graph $G$, we identify the faces of $G$ not only with the vertices in the dual graph $G^*$ but also with the cycles in the boundaries of the faces provided that $G$ is 2-connected and is not a cycle.

Let $T$ be a tree. For $vw \in E(T)$, we denote by $T[vw; v]$ the connected component of $T - vw$ containing $v$. Given a vertex $a \in V(T)$, we specify the tree $T$ rooted at $a$ by $T^{(a)}$. For $v \in V(T)$, $T^{(a)}_v$ is defined as the subtree of $T$ containing $v$ and all of its descendants in $T^{(a)}$.

A graph $G$ is said to be $\kappa$-connected for some $\kappa \in \mathbb{N}$ if $G$ has at least $\kappa + 1$ vertices and $G - U$ is connected for any $U \subseteq V(G)$ with $|U| < \kappa$.

\section{Find Subtrees of Specified Weight} \label{sec:main}

This section is devoted to the proof of Lemma~\ref{lem:ksubtree}. To this end, we \emph{may} assume that it doesn't hold and then deduce a contradiction. However it would be hopeless to derive an efficient algorithm from such a proof, since we would then have to search over possibly exponentially many subtrees of a tree. Fortunately, we can actually consider a linear-size subclass of subtrees instead, which we define in Section~\ref{sec:DFS}. Since then we assume, towards a contradiction, that there is no subtree of $T$ having the desired weight in this subclass. We outline how to prove Lemma~\ref{lem:ksubtree} by double counting in Section~\ref{sec:oda}. We introduce a helpful notion called support subtree in Section~\ref{sec:support}, and study the subtrees in the aforementioned linear-size subclass in Section~\ref{sec:study}. We give a proof of Lemma~\ref{lem:ksubtree} and a pseudocode of a linear-time algorithm in Section~\ref{subsec:proof}. In Section~\ref{sec:examples} we present some examples showing that the conditions in Lemma~\ref{lem:ksubtree} are tight from several aspects.

\subsection{An Overload-Discharge Approach} \label{sec:ODa}

\subsubsection{Overloading Subtrees by ETT} \label{sec:DFS}

To define the subclass of subtrees mentioned above, we consider the subtrees collected by the so-called \emph{Euler tour technique} (ETT) which was first introduced by Tarjan and Vishkin~\cite{Tarjan1984} and has abundant applications in computing and data structures.

We assume a fixed planar embedding of the tree $T$ and we walk around it, i.e. we see edges of $T$ as walls perpendicular to the plane and we walk on the plane along the walls. This walk yields a cycle of size $2(|V(T)| - 1)$.

To make it precise, we define the auxiliary directed cycle graph $C_T$ as follows. For each $v \in V(T)$, we enumerate the edges incident to $v$ in the clockwise order according to the planar embedding and denote them by $e_{v, 1}, e_{v, 2}, \dots, e_{v, d_T(v)}$. The vertex set $V(C_T)$ consists of $d_T(v)$ vertices $w_{v, 1}, w_{v, 2}, \dots, w_{v, d_T(v)}$ for each $v \in V(T)$. Let $w_{v, d_T(v) + 1} := w_{v, 1}$. And, for every edge $uv \in E(T)$, say $uv = e_{u, i} = e_{v, j}$ for some $i \in \{1, \dots, d_T(u)\}$ and $j \in \{1, \dots, d_T(v)\}$, $E(C_T)$ contains the edges $w_{u, i} w_{v, j + 1}$ and $w_{v, j} w_{u, i + 1}$. It is clear that $C_T$ is our desired cycle of size $2(|V(T)| - 1)$ (see Figure~\ref{fig:CT}).

Note that a directed path in $C_T$ can be naturally corresponded to a subtree in $T$. Moreover, growing a subtree by this walking-around-walls in $T$ is equivalent to growing a directed path in $C_T$.

We define the mapping $\rho$ (a homomorphism) from $V(C_T)$ to $V(T)$ by $\rho(w_{v, i}) := v$ for $w_{v, i} \in V(C_T)$ with $v \in V(T)$ and $i \in \{1, \dots, d_T(v)\}$. We also extend this mapping for paths $[u, v]$ directed from $u$ to $v$ in $C_T$ ($u, v \in V(C_T)$) by defining $\rho([u, v]) := T[\{\rho(w) : w \in V([u, v])\}]$. We then extend the weight function $c$ to the vertices $w$ and directed paths $[u, v]$ in $C_T$ ($w, u, v \in V(C_T)$) by $c(w) := c(\rho(w))$ and $c([u, v]) := c(\rho([u, v]))$.

Here we state an assumption (towards a contradiction) which we adopt from now on: \begin{itemize} [\textbf{Assumption \setword{($\Omega$)}{ass:ksubtree}.}]
	\item There are no $x, y \in V(C_T)$ with ${k} - g + 1 \leq c([x, y]) \leq {k}$.
\end{itemize} In other words there is no subtree of $T$ with weight between ${k} - g + 1$ and ${k}$ can be found by searching along the Euler tour. Once we show that Assumption~\ref{ass:ksubtree} cannot hold, we can assure that there exists a subtree $S$ of $T$ having weight ${k} - g + 1 \leq c(S) \leq {k}$. Indeed, we can have some linear-size subclass of subtrees that contains some subtree $S$ having weight ${k} - g + 1 \leq c(S) \leq {k}$. For instance, if $k < c(T)$, we can consider the subtrees corresponding to the paths $[u, v]$ ($u, v \in V(C_T)$) each satisfies $c([u, v]) \le k$, $c([u^-, v]) > k$ and $c([u, v^+]) > k$. There are at most $|V(C_T)| = O(|V(T)|)$ such subtrees, and at least one of them is of weight between $k - g + 1$ and $k$ when Assumption~\ref{ass:ksubtree} doesn't hold. This helps us to devise a linear-time algorithm (Algorithm~\ref{alg:overloaddischarge}) for searching a subtree of the desired weight.

By Assumption~\ref{ass:ksubtree}, the inequalities $c([u, v]) \ge k - g + 1$ and $c([u, v]) \le k$ ($u, v \in V(C_T)$) are equivalent to $c([u, v]) > k$ and $c([u, v]) < k - g + 1$, respectively. (Such usage of Assumption~\ref{ass:ksubtree} would occur tacitly.) 

We mention some more consequences that follow from Assumption~\ref{ass:ksubtree}. It is readily to see that $c(v) \leq k - g$ for any $v \in V(T)$ and $N_2 > k$. Consider a path $[u, v]$ ($u, v \in V(C_T)$) satisfying $c([u, v]) \ge k$, $c([u^+, v]) \le k$ and $c([u, v^-]) \le k$, equivalently, $c([u, v]) > k$, $c([u^+, v]) < k - g + 1$ and $c([u, v^-]) < k - g + 1$. It is clear that such a path exists, $\rho(u) \neq \rho(v)$, and both $c(u)$, $c(v)$ are at least $g + 1$. In particular, we have $\sum_{i \geq g + 1} |V_i| \geq 2$.

Let $u, v \in V(C_T)$. $[u, v]$ is \emph{$k$-overloading} or simply \emph{overloading} if $c([u^-, v^-]) > k$, $c([u, v^-]) \le k$ and $c([u, v]) > k$, and we say $\rho([u, v])$ is an \emph{overloading subtree}. Note that an overloading path always exists when we assume~\ref{ass:ksubtree}, since we are given that $1 \le k < N_2 = c(T)$ and $c(v) \le k$ for every $v \in V(T)$.

\begin{figure}[!ht]
	\centering
	\subfloat[The tree $T$ with vertex weights. \label{subfig:vertexweight2}]{%
		\includegraphics[scale=1]{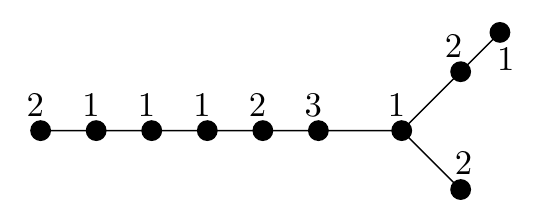}
	}
	\hspace{1.5cm}
	\subfloat[The tree $T$, and the auxiliary cycle $C_T$ whose edges are directed clockwise. Set $k := 7$ and $g := 1$. $Q_{x, y}$ is a maximal overload-discharge quadruples, since $c({[}x, y{]} )> 7$, $c({[}x, y^-{]}) < 7$ and $c({[}x^-, y^-{]}) > 7$. The path ${[}x, y{]}$ and the corresponding overloading subtree are indicated in red and orange, respectively.\label{subfig:overloadingsubtree}]{%
		\includegraphics[scale=1]{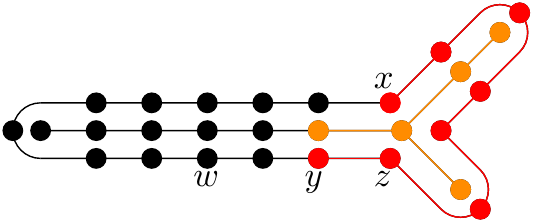}
	}
	\caption{Walk around the tree by ETT.}
	\label{fig:CT}
\end{figure}

\subsubsection{Bounds on $|V_1|$} \label{sec:oda}

Now we see how we can count the number of vertices of weight 1 in two ways so that a contradiction 
may occur. By considering the sum of all the vertex weights, we have that \begin{align*}
N_2 &= \sum_{i \geq 1} i |V_i|\\
&= 2 \sum_{i \geq 1} |V_i| + \sum_{1 \leq i \leq g} (i - 2) |V_i| + \sum_{i \geq g + 1} (i - 2) |V_i|\\
&= 2 N_1 - |V_1| + \sum_{2 \leq i \leq g} (i - 2) |V_i| + \sum_{i \geq g + 1} (i - g - 1) |V_i| + \sum_{i \geq g + 1} (g - 1) |V_i|\\
&= \sum_{i \geq g + 1} (i - g - 1)|V_i| + \sum_{i \geq 2} \min \{i - 2, g - 1\} |V_i| + 2 N_1 - |V_1|.
\end{align*}As $2 N_1 - N_2 =: h$ and $\sum_{i \geq g + 1} |V_i| \geq 2$, we have the following lower bound on $|V_1|$: \begin{align} \label{for:1}
|V_1| \geq \sum_{i \geq g + 1} (i - g - 1)|V_i| + 2g + h - 2. \tag{$\diamond$}
\end{align} The intuitive idea of our proof of Lemma~\ref{lem:ksubtree} is that if $|V_1|$ is large enough, i.e. there are many vertices of weight $1$, then it should facilitate the search of subtree of the desired weight. Therefore, if Lemma~\ref{lem:ksubtree} would not hold, there would be some upper bound on the number of vertices of weight $1$ showing that the inequality (\ref{for:1}) must be contradicted. The upper bound is realized by the following observation.

\begin{observation} \label{obs:1bound}
	Let $g, {k} \in \mathbb{N}$. Let $S$ be a subtree of $T$ with $c(S) > {k}$, $l$ be a leaf of $S$ with $c(S - l) < {k} - g + 1$, $M$ be a subset of $V(S) - l$ and $n$ a vertex in $S - M - l$ such that $S - M$ remains as a tree, $n$ is a leaf of $S - M$, $c(S - M) > {k}$ and $c(S - M - n) < {k} - g + 1$. Then we have \begin{align} \label{for:observation}
	|M \cap V_1| \leq c(S) - ({k} + 1) \leq c(l) - g - 1. \tag{$\ast$}
	\end{align}
\end{observation}

The vertex set $M$ can be seen as a set of vertices which are collected from a leave-cutting process, i.e. we cut leaves (other than $l$) one by one from $S$ with that the weight of the remainder still larger than ${k}$, and it becomes less than ${k} - g + 1$ once we further cut the vertex $n$. Note that $l$ is not cut from the subtree and it always stays as a leaf in the remaining part. 

\begin{proof}
	Since $c(S) - c(M) = c(S - M) > {k}$ and $c(S) - c(l) = c(S - l) < {k} - g + 1$, we have $|M \cap V_1| \leq c(M) \leq c(S) - ({k} + 1)  \leq c(l) - g - 1$.
\end{proof} Note that the conditions given in Observation~\ref{obs:1bound} appear naturally if we have a tree $T$ which has no subtree of weight between ${k} - g + 1$ and ${k}$. We carry out an \emph{overload-discharge process} as follows. We grow a subtree (say a single vertex) which is of weight less than ${k} - g + 1$ until we grow it with a vertex $l$ which makes the weight of the subtree at least ${k} - g + 1$. As we assume that no subtree is of weight between ${k} - g + 1$ and ${k}$, when we halt the growth, the weight of the subtree is actually not only at least ${k} - g + 1$ but greater than ${k}$. We then start to cut its leaves (other than $l$) one by one until the weight declines to be less than ${k} - g + 1$ again. The overload and discharge steps can always be achieved provided that $N_2 > {k}$ and $c(v) \leq {k} - g$ for all $v \in V(T)$. We say that $l$ \emph{overloads} $S$ and a \emph{discharge} $M \cup \{n\}$ containing the \emph{last discharge} $n$ follows, and that $(S, l, M, n)$ is an \emph{overload-discharge quadruple}. It is clear that $l$ and $n$ are two distinct vertices having weight at least $g + 1$.

Let us have a look of a crude argument on how a contradiction would occur. Suppose we have a family of overload-discharge quadruples $(S_f, l_f, M_f, n_f)$ (with some indices $f$) such that the vertices of weight 1 in $T$ is \emph{covered} by the discharges, i.e. $V_1 \subseteq \bigcup_f M_f$, and each overloading vertex $l_f$ corresponds to precisely one overload-discharge quadruple, then, by the inequality~(\ref{for:observation}), we can simply deduce the following contradiction to the inequality (\ref{for:1}):\begin{align*}
|V_1| \leq \sum_f |M_f \cap V_1| \leq \sum_f (c(l_f) - g - 1) \leq \sum_{i \geq g + 1} (i - g - 1)|V_i|.
\end{align*} Although it is not always possible to have such a family of quadruples, we are still able to have some \emph{sufficiently good} family which leads to a contradiction to the inequality~(\ref{for:1}) even if we only assume~\ref{ass:ksubtree}. We will consider the family of overload-discharge quadruples corresponding to overloading subtrees.

We demonstrate how an overload-discharge quadruple can be formed by considering paths in $C_T$. Let $u, v$ be two distinct vertices of $C_T$. If $c([u, v^-]) < {k} - g + 1$ but $c([u, v]) > {k}$, then there exists $w \in V([u, v^-])$ such that $c([w, v]) > {k}$ and $c([w^+, v]) < {k} - g + 1$. It is clear that $\rho(v)$ overloads the subtree $\rho([u, v])$ and we call $$(\rho([u, v]), \rho(v), V(\rho([u, v])) - V(\rho([w, v])), \rho(w)) =: Q_{u, v}$$ an \emph{overload-discharge quadruple associated with $u, v$}.

An overload-discharge quadruple $Q_{u, v}$ associated with $u, v \in V(C_T)$ is \emph{maximal} if $c([u^-, v^-]) > {k}$ holds, or equivalently, $[u, v]$ is an overloading path in $C_T$ (see Figure~\subref*{subfig:overloadingsubtree}). We let $\mathcal{Q}(T; c, {k}) =: \mathcal{Q}$ be the family of all maximal overload-discharge quadruples associated with some $u, v \in V(C_T)$.

\subsection{Support Vertices and Support Subtree} \label{sec:support}

In order to see how the overloading subtrees from $\mathcal{Q}$ would be \emph{packed} in the weighted tree $T$, we need to study its structure in more detail. We introduce the notion of support vertices and support subtree of the weighted tree $T$ in this section. 

We first fix an arbitrary vertex $a \in V(T)$ and consider the rooted tree $T^{(a)}$. Note that there always exists a vertex $r$ such that $c(T^{(a)}_r) > k$ and $c(T^{(a)}_w) \leq k$ for all children $w$ of $r$ in $T^{(a)}$, as we assume $c(T^{(a)}) = N_2 > k$. We then take one such vertex $r$ and consider the tree $T^{(r)}$ rooted at $r$. Let $r_1, \dots, r_t$ ($t \in \mathbb{N}$) be the vertices each satisfies that $c(T^{(r)}_{r_i}) > k$ and $c(T^{(r)}_w) \leq k$ for all children $w$ of $r_i$ in $T^{(r)}$ ($i = 1, \dots, t$). We call $r, r_1, \dots, r_t$ \emph{support vertices} of $T$, and the minimal subtree $T^*$ containing all support vertices \emph{support subtree} of $T$.

Note that $T^{(r)}_w$ is exactly the same subtree as $T^{(a)}_w$ for every $w \in V(T^{(a)}_r) - r$, therefore $c(T^{(r)}_w) = c(T^{(a)}_w) \leq k$ and $r_i \notin V(T^{(a)}_r) - r$ for every $i = 1, \dots, t$. If there are two distinct support vertices $r_{i_1}, r_{i_2}$ with $i_1, i_2 \in \{1, \dots, t\}$, we have that $r_{i_1}$ is neither ancestor nor descendant of $r_{i_2}$ in $T^{(r)}$, and, in particular, both $r_{i_1}, r_{i_2}$ cannot be $r$. It is possible that $r$ is one of the $r_1, \dots, r_t$; it happens if and only if $r$ is the only support vertex and $|V(T^*)| = 1$. We conclude that the leaves of $T^*$ are exactly the support vertices if $|V(T^*)| > 1$, as $T^*$ is the union of the paths $P_i$ ($i = 1, \dots, t$), where $P_i$ is the path in $T$ with endvertices $r$ and $r_i$.

It is clear that $T - E(T^*)$ is a forest of $|V(T^*)|$ subtrees. For $\tilde{r} \in V(T^*)$, we denote by $T^*[\tilde{r}]$ the maximal subtree of $T - E(T^*)$ containing $\tilde{r}$. If $|V(T^*)| > 1$, then the subtrees $T^{(a)}_r, T^{(r)}_{r_1}, \dots, T^{(r)}_{r_t}$ are exactly $T^*[r], T^*[r_1], \dots, T^*[r_t]$, and each of them has weight at least $k + 1$.

Let $t^*$ be the number of support vertices. It is clear that $t^* = t = 1$ if $|V(T^*)| = 1$, and $t^* = t + 1$ if $|V(T^*)| > 1$. We claim that $N_2 \geq t^*(k + 1)$. It holds trivially if $|V(T^*)| = 1$. If $|V(T^*)| > 1$, then we have $t^*$ vertex-disjoint subtrees $T^*[r], T^*[r_1], \dots, T^*[r_t]$. Thus we have $$N_2 = c(T) \geq c(T^*[r]) + c(T^*[r_1]) + \dots + c(T^*[r_t]) \geq t^*(k + 1).$$ In particular, $N_2 \geq 2k + 2$ if $|V(T^*)| > 1$.

We remark that for a fixed $k$ the support tree $T^*$ is unqiuely defined if $|V(T^*)| > 1$, while it is not always uniquely defined if $|V(T^*)| = 1$, as it would depend on the initial root $a$. For ease of presentation we assume that some support tree is fixed throughout.

\subsection{Overloading Vertices and Discharges} \label{sec:study}

In this section we focus on the vertices which overload subtrees from $\mathcal{Q}$ and see how many discharges they can carry each. We first show a sufficient condition for a vertex to be contained in some discharge from $\mathcal{Q}$.

\begin{lemma} \label{lem:indischarge}
	Let $vw$ be an edge in $T$. If $c(T[vw; w]) \geq k - g$, then there exists $(S, l, M, n) \in \mathcal{Q}$ with $v \in M \cup \{n\}$.
\end{lemma}
\begin{proof}
	Let $i \in \{1, \dots, d_T(v)\}$ such that the edge $vw$ is $e_{v, i}$. We grow a path in $C_T$ from $u := w_{v, i}$ to obtain an overload-discharge quadruple $Q_{u, y}$ associated with $u, y$ for some $y \in V(C_T)$. The corresponding situation in $T$ is that a subtree starts growing at $v$, then traverses along the edge $e_{v, i}$ immediately. It will overload, i.e. the weight reaches larger than $k$, without revisiting $v$, since $c(v) + c(T[vw; w]) \geq k - g + 1$.
	
	We can augment the path $[u, y]$ backwards along the cycle $C_T$ to obtain $[x, y]$ such that $c([x, y^-]) < k - g + 1$ but $c([x^-, y^-]) > k$. Then we have $Q_{x, y} \in \mathcal{Q}$. Note that $u$ is the only vertex in $[u, y]$ with $\rho(u) = v$, and $u$ cannot stay in the path after discharge since $c([u, y]) \geq k - g + 1$. Therefore the discharge of $Q_{x, y}$ must contain $v$.
\end{proof}

Now we give a necessary condition for a vertex to be an overloading vertex in some quadruple in $\mathcal{Q}$.

\begin{lemma} \label{lem:olvertex} 
	Let $Q_{x, y} \in \mathcal{Q}$ be an overload-discharge quadruple associated with some $x, y \in V(C_T)$. We have $c(T[\rho(y)\rho(y^-); \rho(y^-)]) + c(\rho(y)) > k$.
\end{lemma}
\begin{proof}
	It is clear that the subtree $\rho([x, y^-])$ is contained in the subtree $T[\rho(y)\rho(y^-); \rho(y^-)]$. Therefore, $c(T[\rho(y)\rho(y^-); \rho(y^-)]) + c(\rho(y)) \geq c([x, y]) > k$ as $\rho(y)$ overloads $\rho([x, y])$.
\end{proof}

We next show that if there are more than one overloading subtrees having the same overloading vertex, then the mutual intersection among these subtrees can only be the overloading vertex, and such a vertex must be in the support subtree. We also prove upper bounds on these discharges.

\begin{lemma} \label{lem:overlap}
	Let $Q_{x_1, y_1}$ and $Q_{x_2, y_2}$ be two distinct overload-discharge quadruples associated with $x_1, y_1$ and $x_2, y_2$, where $x_1, y_1, x_2, y_2 \in V(C_T)$, in $\mathcal{Q}$, respectively. If $\rho(y_1) = \rho (y_2) =: l$, we have $$V(\rho([x_1, y_1])) \cap V(\rho([x_2, y_2])) = \{l\}$$ and $l$ must be a vertex in the support subtree $T^*$.
\end{lemma}
\begin{proof}
	As $Q_{x_1, y_1}$ and $Q_{x_2, y_2}$ are distinct overload-discharge quadruples, by the choice of maximality of elements of $\mathcal{Q}$, $y_1$ must be different from $y_2$. Hence we have distinct indices $i, j \in \{1, \dots, d_T(l)\}$ such that $y_1 = w_{l, i + 1}$ and $y_2 = w_{l, j + 1}$. Note that $y_f$ ($f = 1, 2$) is the only vertex in $[x_f, y_f]$ with $\rho(y_f) = l$ since $l$ is the overloading vertex. It means that $l$ is not in the subtree $\rho([x_f, {y_f}^-])$ ($f = 1, 2$). Moreover, the subtree $\rho([x_1, {y_1}^-])$ is a subtree of $T[e_{l, i}; \rho({y_1}^-)]$, i.e. the component not containing $l$ when deleting the edge $e_{l, i}$, and similarly, $\rho([x_2, {y_2}^-])$ is a subtree of the component not containing $l$ when deleting the edge $e_{l, j}$. Therefore $V(\rho([x_1, {y_1}^-])) \cap V(\rho([x_2, {y_2}^-])) = \emptyset$ and $V(\rho([x_1, y_1])) \cap V(\rho([x_2, y_2])) = \{l\}$. 
	
	Suppose $l \notin V(T^*)$. Let $u \in V(T^*)$ and $w$ be the neighbor of $l$ such that $u \in T[lw; w]$. By the definition of the support subtree, for every neighbor $v$ of $l$ other than $w$, we have $c(T[lv; v]) + c(v) \leq k$. By Lemma~\ref{lem:olvertex}, there are at most one overloading subtree whose overloading vertex is $l$.
\end{proof}

\begin{lemma} \label{lem:hodotiu}
	Let $T_0$ be a subtree of $T$. Let $l$ be an overloading vertex shared by $t \in \mathbb{N}$ quadruples $(S_i, l, M_i, n_i) \in \mathcal{Q}$, for $i = 1, \dots, t$, such that $S_i$ is a subtree of $T_0$ for every $i = 1, \dots, t$. We have \begin{align*}
		\sum_{i = 1}^t |M_i \cap V_1| &\leq c(l) + (t - 2)(c(l) - k - 1) + c(T_0) - 2k - 2 \leq c(T_0) - k - 1.
	\end{align*}
\end{lemma}
\begin{proof}
	By Lemma~\ref{lem:overlap}, the overloading subtrees share only the overloading vertex $l$, hence $c(l) + \sum_i (c(S_i) - c(l)) \leq c(T_0)$. By Observation~\ref{obs:1bound} and the assumption $c(v) \leq k$ for all $v \in V(T)$, we have $\sum_i |M_i \cap V_1| \leq \sum_i (c(S_i) - (k + 1)) = \sum_i (c(S_i) - c(l)) + t (c(l) - (k + 1)) \leq c(T_0) - c(l) + t (c(l) - (k + 1)) = c(T_0) + c(l) - 2(k + 1) + (t - 2)(c(l) - k - 1) = c(T_0) - k - 1 + (t - 1)(c(l) - k - 1) \leq c(T_0) - k - 1$.
\end{proof}

As the last preparation for the proof of Lemma~\ref{lem:ksubtree} we show that a reasonable portion of vertices will be covered by the discharges from $\mathcal{Q}$.

\begin{lemma} \label{lem:covering41}
	If $|V(T^*)| > 1$, then we have $$\bigcup_{(S, l, M, n) \in \mathcal{Q}} (M \cup \{n\}) = V(T).$$ If $|V(T^*)| = 1$ and $N_2 \geq 2k - 2g - D$ for some $D \in \mathbb{N}$, then we have $|\bigcup_{(S, l, M, n) \in \mathcal{Q}} (M \cup \{n\})| \geq |V(T)| - D$. If $|V(T^*)| = 1$ and $N_2 \geq 2k - 2g$, then we have $$\bigcup_{(S, l, M, n) \in \mathcal{Q}} (M \cup \{n\}) \supseteq V(T) - V(T^*).$$
\end{lemma}
\begin{proof}
	If $|V(T^*)| > 1$, for a vertex $v$ in $T$, we can take a support vertex $u \neq v$ such that $v \notin T^*[u]$. Let $w$ be the vertex adjacent to $v$ such that $u \in T[vw, w]$. We have $c(T[vw; w]) \geq c(T^*[u]) \geq k + 1 \geq k - g$, and hence, by Lemma~\ref{lem:indischarge}, there exists $(S, l, M, n) \in \mathcal{Q}$ such that $v \in M \cup \{n\}$. Thus $\bigcup_{(S, l, M, n) \in \mathcal{Q}} (M \cup \{n\}) = V(T)$.
	
	If $|V(T^*)| = 1$ and $N_2 \geq 2k - 2g - D$ for some $D \in \mathbb{N}$, let $r$ be the only one support vertex. Consider the tree $T^{(r)}$ rooted at $r$. Let $U$ be the set of vertices $v$ with $c(T^{(r)}_v) \geq N_2 - k + g + 1$. For $v \in V(T) - r - U$, let $w$ be the parent of $v$ in $T^{(r)}$, we have $c(T[vw; w]) \geq N_2 - (N_2 - k + g) = k - g$ and hence, by Lemma~\ref{lem:indischarge}, $v$ is covered by some discharge from $\mathcal{Q}$. We can assume that $U$ is not empty (otherwise at most one vertex, namely the root $r$, can be not covered by any discharge from $\mathcal{Q}$).
	
	We consider the subtree $T[U]$ of $T$ induced by $U$. If $T[U]$ has two leaves $v, w$ other than $r$, then $|U| \leq 2 + c(T[U] - v - w) \leq 2 + N_2 - c(T^{(r)}_v) - c(T^{(r)}_w) \leq 2 + N_2 - 2(N_2 - k + g + 1) = - N_2 + 2k - 2g \leq D$. Otherwise $T[U]$ is a path with $r$ as one of the endvertices, say $r v_1 v_2 \dots v_t$ for some integer $t \geq 0$. If $t > D$, then $c(T^{(r)}_{v_{1}}) \geq \sum_{i = 1}^{t - 1} c(v_{i}) + c(T^{(r)}_{v_t}) \geq D + c(T^{(r)}_{v_t}) \geq D + (N_2 - k + g + 1) \geq k - g + 1$ which contradicts the definition of the support subtree $T^*$ as in this case $v_1$ should be in $V(T^*)$. If $t = D$, similarly as above, we have $c(T^{(r)}_{v_1}) \geq k - g$ and hence, by Lemma~\ref{lem:indischarge}, $r$ is covered by some discharge from $\mathcal{Q}$. In any case, we have that there are at most $D$ vertices which are not covered by any discharge from $\mathcal{Q}$, i.e. $|\bigcup_{(S, l, M, n) \in \mathcal{Q}} (M \cup \{n\})| \geq |V(T)| - D$.
	
	If $|V(T^*)| = 1$ and $N_2 \geq 2k - 2g$, let $r$ be the only support vertex and $r \neq v \in V(T)$ be a vertex in $T$. By the definition of the support subtree, we have that $c(T^{(r)}_v) \leq k - g$. Let $w$ be the parent of $v$ in $T^{(r)}$. We have $c(T[vw; w]) = N_2 - c(T^{(r)}_v) \geq (2k - 2g) - (k - g) = k - g$. Therefore $V(T) - r \subseteq \bigcup_{(S, l, M, n) \in \mathcal{Q}} (M \cup \{n\})$.
\end{proof}

\subsection{Proof of Lemma~\ref{lem:ksubtree}} \label{subsec:proof}

In this section we prove Lemma~\ref{lem:ksubtree}. We first consider the case that $N_2 \geq 2k - 2g$. If $|V(T^*)| = 1$, let $r$ be the only support vertex. If $c(r) < g + 1$, then by Lemmas~\ref{lem:covering41} and~\ref{lem:overlap} and the condition that $g + h > 2$, we have $|V_1| \leq \sum_{(S, l, M, n) \in \mathcal{Q}} |M \cap V_1| + 1 \leq \sum_{(S, l, M, n) \in \mathcal{Q}} (c(l) - g - 1) + 1 \leq \sum_{i \geq g + 1} (i - g - 1)|V_i| + 1 < \sum_{i \geq g + 1} (i - g - 1)|V_i| + 2g + h - 2$. Otherwise, $c(r) \geq g + 1$ and $r$ can be an overloading vertex and we apply Lemmas~\ref{lem:hodotiu} (take $T_0 := T$) to bound the corresponding discharges as follows: $|V_1| \leq \sum_{(S, l, M, n) \in \mathcal{Q}, l \neq r} |M \cap V_1| + \sum_{(S, l, M, n) \in \mathcal{Q}, l = r} |M \cap V_1| \leq \sum_{(S, l, M, n) \in \mathcal{Q}, l \neq r} (c(l) - g - 1) + \max \{0, c(r) - g - 1, c(r) + N_2 - 2k - 2\} \leq \sum_{(S, l, M, n) \in \mathcal{Q}, l \neq r} (c(l) - g - 1) + c(r) + g + h - 4 \leq \sum_{i \geq g + 1} (i - g - 1)|V_i| + 2g + h - 3$. The third inequality follows from the condition that $N_2 \leq 2k + g + h - 2$. In any case the inequality~(\ref{for:1}) is contradicted.

	If $|V(T^*)| > 1$, then, by Lemma~\ref{lem:covering41}, all vertices in $V_1$ are covered by some discharge from $\mathcal{Q}$. For a vertex $u \in V(T^*)$, by Lemma~\ref{lem:hodotiu} (take $T_0 := T^*[u]$) and Observation~\ref{obs:1bound}, we have $$\sum_{(S, u, M, n) \in \mathcal{Q}} |M \cap V_1| \leq \max \{0, c(T^*[u]) - k - 1\} + d_{T^*}(u) \max \{0, c(u) - g - 1\}.$$ Define $U_1$ to be the set of vertices $u \in V(T^*)$ satisfying $d_{T^*}(u) = 1$, $U_2$ the set of vertices $u \in V(T^*)$ satisfying $d_{T^*}(u) > 1$ and $c(T^*[u]) \geq k + 1$, and $U_3$ the set of vertices $u \in V(T^*)$ satisfying $c(u) \geq g + 1$. Recall that $U_1$ is exactly the set of support vertices and $c(T^*[u]) \geq k + 1$ for all $u \in U_1$. As $U_1$ is disjoint with $U_2$, we have $N_2 \geq \sum_{u \in U_1 \cup U_2} c(T^*[u]) + \sum_{u \in U_3 - (U_1 \cup U_2)} c(u)$, and
	\begin{align*}
		&\sum_{(S, u, M, n) \in \mathcal{Q}, u \in V(T^*)} |M \cap V_1|\\
		\leq &\sum_{u \in U_1 \cup U_2} (c(T^*[u]) - k - 1) + \sum_{u \in U_3} d_{T^*}(u) (c(u) - g - 1)\\
		= &\sum_{u \in U_1 \cup U_2} (c(T^*[u]) - k - 1) + \sum_{u \in U_3} (c(u) - g - 1) + \sum_{u \in U_3 - U_1} (d_{T^*}(u) - 1) (c(u) - g - 1)\\
		\leq &\sum_{u \in U_1 \cup U_2} c(T^*[u]) + \sum_{u \in U_3 - (U_1 \cup U_2)} c(u) + \sum_{u \in U_3} (c(u) - g - 1) + \sum_{u \in U_1} (- k - 1)\\
		&+ \sum_{u \in U_3 \cap U_2} ((d_{T^*}(u) - 1) (c(u) - g - 1) - k - 1) + \sum_{u \in U_3 - (U_1 \cup U_2)} ((d_{T^*}(u) - 1) (c(u) - g - 1) - c(u))\\
		\leq & \hspace{0.06cm}N_2 + \sum_{u \in U_3} (c(u) - g - 1) + \sum_{u \in U_3 - U_1} (d_{T^*}(u) - 2)(- k - 1) + 2(- k - 1)\\
		&+ \sum_{u \in U_3 - U_1} (d_{T^*}(u) - 2) (c(u) - g - 1)\\
		\leq &\sum_{u \in U_3} (c(u) - g - 1) + N_2 - 2k - 2.
	\end{align*} In the third inequality we utilize the basic fact about tree that $\sum_{u \in U_1} 1 = \sum_{u \in U_1} d_{T^*}(u) = \sum_{u \in V(T^*) - U_1} (d_{T^*}(u) - 2) + 2$.
	Thus we have \begin{align*}
		|V_1| &\leq \sum_{(S, l, M, n) \in \mathcal{Q}, l \notin V(T^*)} |M \cap V_1| + \sum_{(S, l, M, n) \in \mathcal{Q}, l \in V(T^*)} |M \cap V_1|\\
		&\leq \sum_{(S, l, M, n) \in \mathcal{Q}, l \notin V(T^*)} (c(l) - g - 1) + \sum_{l \in U_3} (c(l) - g - 1) + N_2 - 2k - 2\\
		&\leq \sum_{i \geq g + 1} (i - g - 1)|V_i| + g + h - 4,
	\end{align*} which contradicts the inequality~(\ref{for:1}).

	We now consider the case that $N_2 < 2k - 2g$. Note that in this case $|V(T^*)| = 1$ always holds. Let $r$ be the only support vertex. Set $D := 2g + h - 3 > 0$ in Lemma~\ref{lem:covering41}, we have
	\begin{align*}
		|V_1| &\leq \sum_{(S, l, M, n) \in \mathcal{Q}, l \neq r} (c(l) - g - 1) + \max \{0, c(r) - g - 1, c(r) + N_2 - 2k - 2\} + D\\
		&\leq \sum_{(S, l, M, n) \in \mathcal{Q}, l \neq r} (c(l) - g - 1) + \max \{0, c(r) - g - 1, c(r) - 2g - 3\} + 2g + h - 3\\
		&\leq \sum_{i \geq g + 1} (i - g - 1)|V_i| + 2g + h - 3,
	\end{align*} which contradicts the inequality~(\ref{for:1}).
	
	Thus it is proved the existence of a subtree $S$ with weight $k - g + 1 \leq c(S) \leq k$. As Assumption~\ref{ass:ksubtree} cannot hold, it is not hard to see that the subtree $S$ can be found by the iterative overload-discharge process described in Algorithm~\ref{alg:overloaddischarge}. The cycle $C_T$ and the mapping $\rho$ can be constructed in $O(N_1)$ time~\cite{Hierholzer1873}. Note that there are $O(|V(C_T)|)$, i.e. $O(N_1)$ iterations, since the initial vertex $v \in V(C_T)$ can be revisited at most once. Thus $S$ can be computed in $O(N_1)$ time. This completes the proof of Lemma~\ref{lem:ksubtree}.
	
	For instance, if we set $s := x$ in the example given in Figure~\subref*{subfig:overloadingsubtree}, Algorithm~\ref{alg:overloaddischarge} will output the subtree $\rho([z, w])$.
	
	We remark that here we assume that each arithmetic operation  be done in constant time. If the arithmetic operations require logarithmic cost, then one can have $O(\Delta(T) \cdot N_1 \log \frac{N_2}{N_1})$ running time, where $\Delta(T)$ denotes the maximum degree of $T$.
	
	\smallskip

\begin{algorithm}[H] \label{alg:overloaddischarge}
	\SetAlgoCaptionSeparator{}
	\caption{}
	\DontPrintSemicolon
	\SetKw{KwGoTo}{go to}
	\KwIn{A tree $T$ of $N_1$ vertices and vertex weights $c : V(T) \rightarrow \mathbb{N}$ with $c(T) = N_2$ such that $1 \leq k \leq N_2$, $g + h > 2$ and $2k - 4g - h + 3 \leq N_2 \leq 2k + g + h - 2$ ($k, g, N_1, N_2 \in \mathbb{N}$), where $h := 2N_1 - N_2$, and $c(v) \le k$ for any $v \in V(T)$.}
	\KwOut{A subtree $S$ of $T$ with $k - g + 1 \leq c(S) \leq k$.}
	Construct the directed cycle $C_T$ and the homomorphism $\rho: V(C_T) \rightarrow V(T)$. Choose an arbitrary vertex $v$ of $C_T$. Set $s:= v$ and $t := v$.\\
	\While{$c([s, t]) < k - g + 1$}{ \label{od}
		Set $t := t^+$.
	}
	\If{$c([s, t]) \leq k$}{
		Output $\rho([s, t])$.
	}
	\While{$c([s, t]) > k$}{ 
		Set $s := s^+$.
	}
	\If{$c([s, t]) \geq k - g + 1$}{
		Output $\rho([s, t])$.
	}
	\KwGoTo \scriptsize \bfseries \ref{od}.\\
\end{algorithm}

\subsection{Some Examples} \label{sec:examples}

In this section we give some examples and show that the conditions in Lemma~\ref{lem:ksubtree} are tight from several aspects. A useful fact to study some examples mentioned below is that Algorithm~\ref{alg:overloaddischarge} is an exhaustive search when the input tree is a path.

The condition $1 \le k \le N_2$ should clearly be included for our interest.

We show that the condition $g + h > 2$ is tight. Consider the \emph{star} $T$ of order $2p$ for some $p > 1$, such that center vertex has weight 1 and the other $2p - 1$ vertices have weight 2. We have $N_2 = 2N_1 - 1 = 4p - 1$ and $h = 1$. Set $k := N_1 = 2p \ge 4$ and $g := 1$. One can easily check that all conditions are satisfied except that $g + h = 2$, and $T$ has no subtree of weight $k$.

For the condition $N_2 \ge 2k - 4g - h + 3$, we consider the path $v_1 v_2 \dots v_{p + 2q}$ of order $p + 2q$ for some integers $p > 1$ and $q \ge 1$. Set $c(v_{q + i}) := 1$ for $i = 1, 2, \dots, p$, and $c(v_j) := 2$ for any $j \neq q + 1, q + 2, \dots, q + p$. We have $N_2 = p + 4q$ and $h = p$. Set $k := p + 2q + 1$ and $g := 1$. One can easily check that all conditions are satisfied except that $N_2 = 2k - 4g - h + 2$, and $T$ has no subtree of weight $k$.

We next discuss the condition $N_2 \le 2k + g + h - 2$. Let $p > 1$ be an integer, and $T$ be a path $v_1 v_2 \dots v_{2p + 3}$ of order $2p + 3$. Set $c(v_{p + 2}) := p + 2$, $c(v_{p + i}) := 2$ for $i = 1, 3$, and $c(v_j) := 1$ for any $j \neq p + 1, p + 2, p + 3$. We have $N_2 = 3p + 6$ and $h = 2N_1 - N_2 = p$. Set $k := p + 3$ and $g := 1$. One can easily check that all conditions are satisfied except that $N_2 = 2k + g + h - 1$, and $T$ has no subtree of weight $k$.

As we have seen, the condition $c(v) \le k$ for all $v \in V(T)$ is one of the key ingredients to make the overload-discharge process work. If this condition is violated, then the existence of a subtree of the desired weight cannot be assured. Let $T$ be the star of order $p + 1$ for some integer $p > 1$. We set the vertex weight of the center vertex to be $q + 1$ for some integer $2 < q < p + 2$, and those of the other $p$ vertices (leaves) to be 1. We have $N_2 = p + q + 1$ and $h = p - q + 1$. Set $k := q$ and $g := 2$. Then it is clear that all conditions are satisfied except that there exists a vertex (the center vertex) with weight larger than $k$, and $T$ has no subtree of weight between $k - g + 1$ and $k$.

\section{Cycle Spectra of Planar Graphs} \label{sec:cycle}

The \emph{cycle spectrum} ${CS}(G)$ of a graph $G$ is defined to be the set of integers $k$ for which there is a cycle of length $k$ in $G$. $G$ is said to be \emph{hamiltonian} if $|V(G)| \in CS(G)$ and \emph{pancyclic} if its cycle spectrum has all possible lengths, i.e. $CS(G) = \{3, \dots, |V(G)|\}$.
Cycle spectra of graphs have been extensively studied in many directions, in this paper we study cycle spectra of planar hamiltonian graphs with minimum degree $\delta \geq 4$. We first give an overview of previous results in Section~\ref{subsec:cycleconj} and present our results in Sections~\ref{sec:Mohr} and~\ref{sec:halfminus1}.

\subsection{An Overview} \label{subsec:cycleconj}

In 1956, Tutte~\cite{Tutte1956} proved his seminal result that every $4$-connected planar graph is hamiltonian. Motivated by Tutte's theorem together with the metaconjecture proposed by Bondy~\cite{Bondy1975} that almost any non-trivial conditions
for hamiltonicity of a graph should also imply pancyclicity, Bondy~\cite{Bondy1975} conjectured in 1973 that every $4$-connected planar graph $G$ is almost pancyclic, i.e. $|CS(G)| \geq |V(G)| - 3$, and Malkevitch~\cite{Malkevitch1988} conjectured in 1988 that every $4$-connected planar graph is pancyclic if it
contains a cycle of length $4$ (see~\cite{Malkevitch1971, Malkevitch1978} for other variants).

These two conjectures remain open, while $4$ is the only known cycle length that can be missing in a cycle spectrum of a $4$-connected planar graph. For example, the line graph of a cyclically $4$-edge-connected cubic planar graph with girth at least $5$ is a $4$-regular $4$-connected planar graph with no cycle of length $4$, see also~\cite{Malkevitch1971, Trenkler1989}. If we relax the connectedness, more cycle lengths are known for being absent in some cycle spectra. Choudum~\cite{Choudum1977} showed that for every integer $k \geq 7$, there exist $4$-regular $3$-connected planar hamiltonian graphs of order larger than $k$ each has cycles of all possible lengths except $k$, which means every integer $k \geq 7$ can be absent in the cycle spectra of some $4$-regular $3$-connected planar hamiltonian graphs. Another interesting example was given by Malkevitch~\cite{Malkevitch1971}, which is, for every $p \in \mathbb{N}$, a $4$-regular planar hamiltonian graph $G$ of order $|V(G)| = 6p$ whose cycle spectrum $CS(G) = \{3, 4, 5, 6\} \cup \{r \in \mathbb{N}: \frac{|V(G)|}{2} \leq r \leq |V(G)|\}$, as shown in Figure~\ref{fig:Malkevitch}.

\begin{figure}[b!] 
	\centering
	\includegraphics[scale = 1]{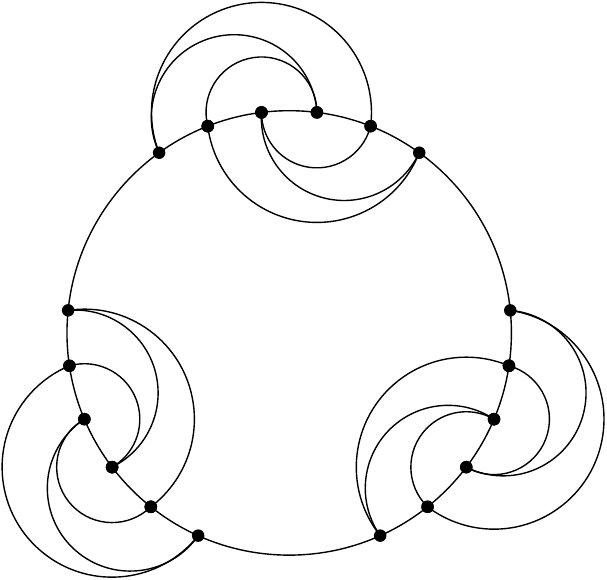}
	\caption{A $4$-regular planar hamiltonian graph $G$ which has no cycle of length between $7$ and $\frac{|V(G)|}{2} - 1$.}
	\label{fig:Malkevitch}
\end{figure}

So far we have seen which cycle lengths can be absent in some cycle spectra, we now ask the opposite question, i.e. which cycle lengths must be present in all cycle spectra. It is known that every planar graph with $\delta \geq 4$ must contain cycles of length $3$, $5$~\cite{Wang2002} and $6$~\cite{Fijavz2002}, which is shown to be best possible by the aforementioned examples. It is also known that every $2$-connected planar graph with $\delta \geq 4$ must have a cycle of length $4$ or $7$~\cite{Hornak2008}, a cycle of length $4$ or $8$ and a cycle of length $4$ or $9$~\cite{Madaras}. While the presence of a cycle of length $3$ follows easily from Euler's formula, the rest of them were shown by the discharging method.

Another powerful tool in searching cycles of specified length is the so-called Tutte path method, which was first introduced by Tutte in his proof of hamiltonicity of 4-connected planar graphs. Using this technique, Nelson (see~\cite{Plummer1975, Thomassen1983}), Thomas and Yu~\cite{Thomas1994} and Sanders~\cite{Sanders1996} showed that every $4$-connected planar graph contains cycles of length $|V(G)| - 1$, $|V(G)| - 2$ and $|V(G)| - 3$, respectively. Note that we always assume $k \geq 3$ when we say a graph contains a cycle of length $k$. Chen et al.~\cite{Chen2004} noticed that the Tutte path method cannot be generalized for smaller cycle lengths, they hence combined Tutte paths with contractible edges and showed the existence of cycles of length $|V(G)| - 4$, $|V(G)| - 5$ and $|V(G)| - 6$. Following this approach, Cui et al.~\cite{Cui2009} showed that every $4$-connected planar graph has a cycle of length $|V(G)| - 7$. To summarize, every $4$-connected planar graph contains a cycle of length $k$ for every $k \in \{|V(G)|, |V(G)| - 1, \dots, |V(G)| - 7\}$ with $k \geq 3$.

\subsection{Cycles of Length Close to Medium-Length} \label{sec:Mohr}

With the knowledge of these short and long cycles, Mohr \cite{Mohr2018} asked whether cycles of length close to $\frac{|V(G)|}{2}$ also exist, and he answered his question by showing that every planar hamiltonian graph $G$ satisfying $|E(G)| \geq 2|V(G)|$ has a cycle of length between $\frac{1}{3}|V(G)|$ and $\frac{2}{3}|V(G)|$. We present his simple and elegant argument in the following.

Let $G^*$ be the dual graph of the plane graph $G$ and $C$ be a Hamilton cycle of $G$. Note that $C$ separates the Euclidean plane into two open regions $C_{\textrm{int}}$ and $C_{\textrm{ext}}$ containing no vertex. Let $G_{\textrm{int}}$ and $G_{\textrm{ext}}$ be the graphs obtained from $G$ by deleting the edges in $C_{\textrm{ext}}$ and in $C_{\textrm{int}}$, respectively. We always assume that $|E(G_{\textrm{int}})| \geq |E(G_{\textrm{ext}})|$. As $C$ is a Hamilton cycle, its dual disconnects $G^*$ into two trees say $T_{\textrm{int}}$ lying on $C_{\textrm{int}}$ and $T_{\textrm{ext}}$ on $C_{\textrm{ext}}$ (see Figure~\ref{fig:transformation}). By Euler's formula, we have $|V(G^*)| = |E(G)| - |V(G)| + 2$, and hence $|V(T_{\textrm{int}})| \geq \frac12|V(G^*)| \geq \frac{1}{2}|V(G)| + 1$. We define vertex weight $c(v) := d_{G^*}(v) - 2 \geq 1$ for every vertex $v \in V(G^*) \supset V(T_{\textrm{int}})$, where $d_{G^*}(v)$ is the degree of $v$ in $G^*$, or equivalently, the face length of $v$ in $G$ (see Figure~\subref*{subfig:vertexweight}). It is not hard to see that for every subtree $S$ of $T_{\textrm{int}}$, the set of edges of $G^*$ having exactly one endvertex in $S$ is indeed  the dual of an edge set of a cycle in $G$ of length $c(S) + 2$, where $c(S) := \sum_{v \in V(S)} c(v)$ (see Figure~\subref*{subfig:treeandcycle}).

\begin{figure}[!ht]
	\centering
	\subfloat[The tree $T_{\textrm{int}}$ with vertex weights. \label{subfig:vertexweight}]{%
		\includegraphics[scale=1.2]{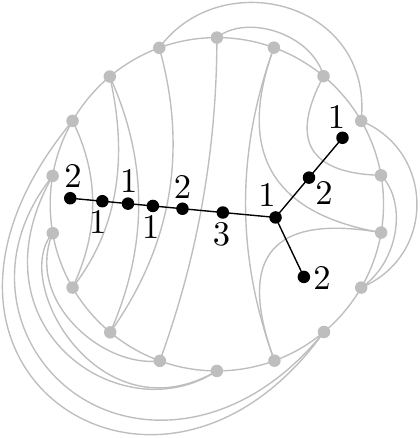}
	}
	\hspace{1.5cm}
	\subfloat[A subtree (red) of $T_{\textrm{int}}$ of weight $6$ corresponds to a cycle (blue) of length $8$ in $G$.\label{subfig:treeandcycle}]{%
		\includegraphics[scale=1.2]{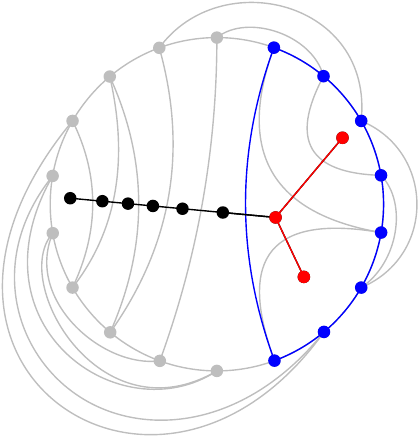}
	}
	\caption{The tree $T_{\textrm{int}}$ (black) in the dual graph of $G$ (grey).}
	\label{fig:transformation}
\end{figure}

Thus the problem is transformed to that of finding a cycle of specified length in a hamiltonian outerplanar graph with sufficient edge density, dually, a subtree of specified weight: the existence of a subtree $S$ of weight $k$ in $T_{\textrm{int}}$ implies the existence of a cycle of length $k + 2$ in $G$. It is left to show that there is a subtree $S$ in $T_{\textrm{int}}$ with $\frac{1}{3}|V(G)| - 2 \leq c(S) \leq \frac{2}{3}|V(G)| - 2$. First note that $c(v) \leq \frac12|V(G)| - 2$ for all $v \in V(T_{\textrm{int}})$; otherwise $c(T_{\textrm{int}}) >\frac12|V(G)| - 2 + |V(T_{\textrm{int}})| - 1 \geq |V(G)| - 2$, which is not possible as $T_{\textrm{int}}$ corresponds to the Hamilton cycle of length $|V(G)|$. If there is a vertex $v \in V(T_{\textrm{int}})$ with $c(v) \geq \frac13|V(G)| - 2$, then we can simply take $S$ to be this single vertex $v$. Suppose $c(v) < \frac13|V(G)| - 2$ for all $v \in V(T_{\textrm{int}})$. We take $S$ to be a maximal subtree of $T_{\textrm{int}}$ with $c(S) \leq \frac23|V(G)| - 2$, it is clear that $c(S) \geq \frac13|V(G)| - 2$. Thus $G$ has a cycle of length between $\frac{1}{3}|V(G)|$ and $\frac{2}{3}|V(G)|$.

We recapitulate the main content of Mohr's proof. Given a planar hamiltonian graph $G$, we can have a tree $T$ (in the dual graph) of at least $\frac12|E(G)| - \frac12|V(G)| + 1$ vertices with vertice weights $c: V(T) \rightarrow \mathbb{N}$ such that $c(T) = |V(G)| - 2$ and $c(v) \leq c(T) - |V(T)| + 1 \leq \frac32|V(G)| - \frac12|E(G)| - 2$ for all $v \in V(T)$. And, if there is a subtree of weight $k$ in $T$, then there is a cycle of length $k + 2$ in $G$. By Lemma~\ref{lem:ksubtree} we have the following.
\begin{theorem} \label{cor:EgeqV}
	Let $G$ be a planar hamiltonian graph with $|E(G)| \geq (2 + \gamma)|V(G)|$ for some real number $-1 \leq \gamma < 1$. Let $k, g \in \mathbb{N}$ such that $g + \lceil \gamma |V(G)| \rceil + 2 > 0$, $3 \leq k \leq |V(G)|$ and $\lfloor \frac{(1 - \gamma)|V(G)|}{2} \rfloor \leq k \leq \frac{\lceil (1 + \gamma)|V(G)| \rceil}{2} + 2g + \frac32$. There exists a cycle $K$ in $G$ of length $k - g + 1 \leq |V(K)| \leq k$, and $K$ can be found in linear time if a Hamilton cycle of $G$ is given.
\end{theorem}
\begin{proof}
	Let $T$ be the tree with vertex weights $c$ that we mentioned before. We set $\tilde{k} := k - 2 \geq 1$ and $h := \lceil \gamma |V(G)| \rceil + 4$. We check the conditions required for applying Lemma~\ref{lem:ksubtree} on the parameters $\tilde{k}, g, h, N_1$ and $N_2$ as follows. First we have that $g + h > 2$, $1 \leq \tilde{k} \leq N_2 = |V(G)| - 2$, $2\tilde{k} \leq \lceil (1 + \gamma) |V(G)| \rceil + 4g - 1 = N_2 + 4g + h - 3$, and $\tilde{k} \geq \lfloor \frac{(1 - \gamma) |V(G)|}{2} \rfloor - 2 \geq \frac{\lfloor (1 - \gamma) |V(G)| \rfloor}{2} - \frac12 - 2 = \frac{|V(G)| - \lceil \gamma |V(G)| \rceil}{2} - \frac12 - 2 \geq \frac{N_2}{2} - \frac{g}{2} - \frac{h}{2} + 1$. Note also that $2|V(T)| \geq |E(G)| - |V(G)| + 2 \geq (1 + \gamma) |V(G)| + 2 = c(T) + \gamma |V(G)| + 4$ implies $2N_1 \geq N_2 + h$, and for every $v \in V(T)$, $c(v) \leq \frac32 |V(G)| - \frac12 |E(G)| - 2 \leq \frac32 |V(G)| - \frac{2 + \gamma}{2} |V(G)| - 2 = \frac{1 - \gamma}{2} |V(G)| - 2$ implies $c(v) \leq \lfloor \frac{(1 - \gamma) |V(G)|}{2} \rfloor - 2 \leq \tilde{k}$. As all conditions are satisfied, by Lemma~\ref{lem:ksubtree}, there exists a subtree $S$ of $T$ of weight $\tilde{k} - g + 1 \leq c(S) \leq \tilde{k}$ which can be found in linear time. And hence $G$ has a cycle $K$ of length $k - g + 1 \leq |V(K)| \leq k$ which can be found in linear time provided a Hamilton cycle of $G$ is given, since every planar graph can be embedded in plane in linear time~\cite{Chiba1985} and the tree $T_{\textrm{int}}$ can then be easily constructed from the planar embedding in linear time.
\end{proof}

We specify some implications as follows.

\begin{corollary}
	Every planar hamiltonian graph $G$ with $\delta(G) \geq 3$ has a cycle of length $\lfloor \frac{|V(G)|+1}{4} \rfloor + 2 \leq k \leq \lfloor \frac{3|V(G)|}{4} \rfloor$. Every planar hamiltonian graph $G$ with $\delta(G) \geq 4$ has a cycle of length $k$ for every $k \in \{\lfloor \frac{|V(G)|}{2} \rfloor, \dots, \lceil \frac{|V(G)|}{2} \rceil + 3\}$ with $3 \leq k \leq |V(G)|$. Every planar hamiltonian graph $G$ with $\delta(G) \geq 5$ has a cycle of length $k$ for every $k \in \{\lfloor \frac{|V(G)|}{4} \rfloor, \dots, \lceil \frac{3|V(G)|}{4} \rceil + 3\}$ with $3 \leq k \leq |V(G)|$. Each of these cycles can be found in linear time if a Hamilton cycle of $G$ is given.
\end{corollary}
\begin{proof}
	It follows immediately when we in Theorem~\ref{cor:EgeqV} set $\gamma := - \frac12$, $g := \lfloor \frac{|V(G)|}{2} \rfloor - 1$ and $k := \lfloor \frac{3|V(G)|}{4} \rfloor$; $\gamma := 0$ and $g := 1$; and $\gamma := \frac12$ and $g := 1$, respectively.
\end{proof}

It is known that a Hamilton cycle can be found in linear time for every 4-connected planar graph~\cite{Chiba1989}, thus those cycles mentioned above can be simply found in linear time each in this case.

\subsection{3-Connected Planar Hamiltonian Graphs} \label{sec:halfminus1}

Note that Malkevitch's example (see Figure~\ref{fig:Malkevitch}) illustrates that not every planar hamiltonian graph $G$ with $\delta \geq 4$ can have a cycle of length $\lfloor \frac{|V(G)|}{2} \rfloor - 1$ or $\lfloor \frac{|V(G)|}{2} \rfloor - 2$. As a further application we prove in this section that this cycle length can be assured for $3$-connected planar hamiltonian graphs with $\delta \geq 4$.

\begin{theorem}
	Let $G$ be a 3-connected planar hamiltonian graph with minimum degree $\delta(G) \geq 4$. If $|V(G)| \geq 8$ is even, there exists a cycle of length either $\frac12|V(G)| - 2$ or $\frac12|V(G)| - 1$ in $G$, and it can be found in linear time if a Hamilton cycle is given.
\end{theorem}
\begin{proof}
	We adopt the notations defined in Section~\ref{sec:Mohr}. If every face of $G_{\textrm{int}}$ is of length either $|V(G)|$ or less than $\frac12|V(G)|$, i.e. $c(v) \leq \frac12|V(G)| - 3$ for every $v \in V(T_{\textrm{int}})$, by Lemma~\ref{lem:ksubtree} (set $g := 2$ and $h := 4$), there exists a subtree of weight either $\frac12|V(G)| - 4$ or $\frac12|V(G)| - 3$ in $T_{\textrm{int}}$ and hence a cycle of length either $\frac12|V(G)| - 2$ or $\frac12|V(G)| - 1$ in $G$.
	
	Recall that $|E(G_{\textrm{int}})| \geq \frac32|V(G)|$. If $|E(G_{\textrm{int}})| > \frac32|V(G)|$, then $|V(T_{\textrm{int}})| \geq \frac12|V(G)| + 2 = \frac12c(T_{\textrm{int}}) + 3$ and $c(v) \leq c(T_{\textrm{int}}) - |V(T_{\textrm{int}})| + 1 \leq \frac12|V(G)| - 3$ for all $v \in V(T_{\textrm{int}})$. By Lemma~\ref{lem:ksubtree} (set $g := 1$ and $h := 6$), there exists a subtree of weight $\frac12|V(G)| - 3$ in $T_{\textrm{int}}$ and hence a cycle of length $\frac12|V(G)| - 1$ in $G$.
	
	Now we can assume that $|E(G_{\textrm{int}})| = \frac32|V(G)|$ and $G_{\textrm{int}}$ has a face of length $\frac12|V(G)|$. It holds immediately that $|E(G_{\textrm{ext}})| = \frac32|V(G)|$ since $|E(G_{\textrm{int}})| + |E(G_{\textrm{ext}})| = |E(G)| + |V(G)| \geq 3|V(G)|$ and $|E(G_{\textrm{int}})| \geq |E(G_{\textrm{ext}})|$. And we can also assume that $G_{\textrm{ext}}$ has a face of length $\frac12|V(G)|$. In this case we have $d_G(v) = 4$ and $d_{G_{\textrm{int}}}(v) + d_{G_{\textrm{ext}}}(v) = 6$ for every $v \in V(G)$, and that there are exactly one face of length $|V(G)|$, one face of length $\frac12|V(G)|$ and $\frac12|V(G)|$ faces of length $3$ in each of $G_{\textrm{int}}$ and $G_{\textrm{ext}}$. We denote by $F_{\textrm{int}}$ and $F_{\textrm{ext}}$ be the faces of length $\frac12|V(G)|$ in $G_{\textrm{int}}$ and $G_{\textrm{ext}}$, respectively.
	
	We claim that $G$ is the square of a cycle of length $|V(G)|$, which is obtained from a cycle of length $|V(G)|$ by adding edges for every pair of vertices having distance $2$ (see Figrue~\ref{fig:squareC16}). It is obvious that the square of a cycle is pancyclic. We call a face of length $3$ an \emph{$i$-triangle} ($i = 0, 1, 2$) if it contains exactly $i$ edges of the Hamilton cycle $C$. We assume that the plane graph $G$ has the maximum number of $2$-triangles over all of its planar embeddings. Let the Hamilton cycle $C$ of $G$ be $v_0 v_1 v_2 \dots v_{|V(G)| - 1} v_{0}$ (indices modulo $|V(G)|$).

	\begin{figure}[b!] 
		\centering
		\includegraphics[scale = 1]{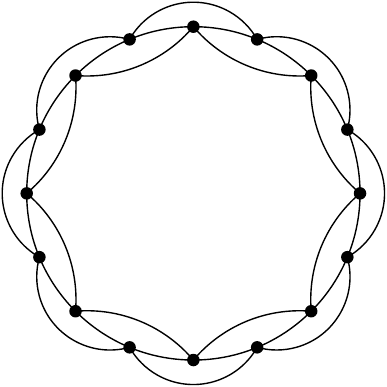}
		\caption{The square of a cycle of length $16$.}
		\label{fig:squareC16}
	\end{figure}
	
	Suppose there is a $0$-triangle $v_0 v_i v_j v_0$ in $G$, say it is also in $G_{\textrm{int}}$, for some $0 < i - 1 < j - 2 < |V(G)| - 3$. If $i > 2$, then the face in $G_{\textrm{int}}$ containing the path $v_1 v_0 v_i v_{i - 1}$ is of length larger than $3$ and smaller than $|V(G)|$. As there is exactly one such face in $G_{\textrm{int}}$, namely $F_{\textrm{int}}$, we can assume that $i = 2$ and $j = 4$. Then $d_{G_{\textrm{int}}}(v_1) = d_{G_{\textrm{int}}}(v_3) = 2$ and $d_{G_{\textrm{ext}}}(v_1) = d_{G_{\textrm{ext}}}(v_3) = 4$. Let $v_{i_1}, v_{i_2}$ be the neighbors of $v_1$ other than $v_0, v_2$, and $v_{i_3}, v_{i_4}$ be the neighbors of $v_3$ other than $v_2, v_4$, for some $2 < i_1 < i_2 < |V(G)|$ and $4 < i_3 < i_4 < |V(G)| + 2$.
	
	If $v_1$ is adjacent to $v_3$ in $G_{\textrm{ext}}$, i.e. $i_1 = 3$ and $i_4 = |V(G)| + 1$, and the face in $G_{\textrm{ext}}$ containing $v_{i_2} v_1 v_3 v_{i_3}$ is a face of length $3$, i.e. $i_2 = i_3$, then it must be a $0$-triangle. We can assume that $i_2 = i_3 = 5$. Clearly, $\{v_0, v_5\}$ is a separator of $G$, which contradicts that $G$ is $3$-connected. If $v_1$ is adjacent to $v_3$ in $G_{\textrm{ext}}$, but the face in $G_{\textrm{ext}}$ containing $v_{i_2} v_1 v_3 v_{i_3}$ is a face of length larger than $3$, then the faces in $G_{\textrm{ext}}$ containing $v_0 v_1 v_{i_2}$ and $v_{i_3} v_3 v_4$ must be $2$-triangles and $\{v_{i_2}, v_{i_3}\} = \{v_{-1}, v_5\}$ is a separator of $G$, contradiction.
	
	If $v_1$ is not adjacent to $v_3$, then the face in $G_{\textrm{ext}}$ containing $v_{i_1} v_1 v_2 v_3 v_{i_4}$ is of length larger than $3$, and hence $v_{-1} v_0 v_1 v_{-1}$ and $v_3 v_4 v_5 v_3$ must be $2$-triangles in $G_{\textrm{ext}}$. In this case we can swap $v_0$ and $v_1$ and swap $v_3$ and $v_4$ to obtain a planar embedding with more $2$-triangles (see Figure~\subref*{subfig:0triangleswap}), which contradicts the maximality of the number of $2$-triangles. Thus there is no $0$-triangle in the plane graph $G$.
	
	Suppose there is a $1$-triangle $v_0 v_1 v_i v_0$ in $G$, say also in $G_{\textrm{int}}$, for some $2 < i < |V(G)| - 1$. It is not hard to see that we can assume that the face in $G_{\textrm{int}}$ containing $v_0 v_i v_{i + 1}$ is $F_{\textrm{int}}$. Under this assumption we must have a sequence of $i - 1$ faces of length 3 such that all faces are $1$-triangles except the last one which is a $2$-triangle, namely $v_0 v_1 v_i v_0, v_1 v_{i - 1} v_i v_1, v_1 v_2 v_{i - 1} v_1, \dots, v_{\lceil \frac{i}{2} \rceil - 1} v_{\lceil \frac{i}{2} \rceil} v_{\lceil \frac{i}{2} \rceil+ 1} v_{\lceil \frac{i}{2} \rceil - 1}$.
	
	We claim that $i \leq 4$. Suppose $i > 4$, we prove the claim for odd $i$, it can be proved for even $i$ in a similar way. It is clear that $d_{G_{\textrm{ext}}}(v_{\lceil \frac{i}{2} \rceil - 3}) \leq 3$, $d_{G_{\textrm{ext}}}(v_{\lceil \frac{i}{2} \rceil - 1}) = 3$ and $d_{G_{\textrm{ext}}}(v_{\lceil \frac{i}{2} \rceil}) = 4$. Let $v_{i_1}, v_{i_2}$ be the neighbors of $v_{\lceil \frac{i}{2} \rceil}$ other than $v_{\lceil \frac{i}{2} \rceil - 1}, v_{\lceil \frac{i}{2} \rceil + 1}$, and $v_{i_3}$ be the neighbor of $v_{\lceil \frac{i}{2} \rceil - 1}$ other than $v_{\lceil \frac{i}{2} \rceil - 2}, v_{\lceil \frac{i}{2} \rceil}, v_{\lceil \frac{i}{2} \rceil + 1}$, for some $i < i_1 < i_2 \leq i_3 \leq |V(G)|$. Note that the face in $G_{\textrm{ext}}$ containing $v_{i_1} v_{\lceil \frac{i}{2} \rceil} v_{\lceil \frac{i}{2} \rceil + 1} v_{\lceil \frac{i}{2} \rceil + 2}$ is of length larger than $3$. Therefore the face in $G_{\textrm{ext}}$ containing $v_{i_3} v_{\lceil \frac{i}{2} \rceil - 1} v_{\lceil \frac{i}{2} \rceil} v_{i_2}$ and that containing $v_{\lceil \frac{i}{2} \rceil - 3} v_{\lceil \frac{i}{2} \rceil - 2} v_{\lceil \frac{i}{2} \rceil - 1} v_{i_3}$ must be of length $3$. It implies that $v_{\lceil \frac{i}{2} \rceil - 3} = v_{i_3} = v_{i_2}$ and $d_{G_{\textrm{ext}}}(v_{\lceil \frac{i}{2} \rceil - 3}) \geq 4$, contradiction.
	
	Now we consider the case when $i = 4$. It is clear that $d_{G_{\textrm{ext}}}(v_2) = 4$ and $d_{G_{\textrm{ext}}}(v_3) = 3$. Let $v_{i_1}$ be the neigbor of $v_3$ other than $v_1, v_2, v_4$, and $v_{i_2}, v_{i_3}$ be the neigbors of $v_2$ other than $v_1, v_3$, for some $4 < i_1 \leq i_2 < i_3 \leq |V(G)|$. If the face in $G_{\textrm{ext}}$ containing $v_{i_1} v_3 v_4$ is of length larger than $3$, then $i_1 = i_2 = |V(G)| - 1$, $i_3 = |V(G)|$ and $\{v_{-1}, v_4\}$ is separator of $G$. If the face in $G_{\textrm{ext}}$ containing $v_{i_1} v_3 v_4$ is of length $3$ but that containing $v_{i_1} v_3 v_2 v_{i_2}$ is of length larger than $3$, then $i_1 = 5$, $i_2 = |V(G)| - 1$, $i_3 = |V(G)|$ and $\{v_{-1}, v_5\}$ is a separator of $G$. If the faces in $G_{\textrm{ext}}$ containing $v_{i_1} v_3 v_4$ and $v_{i_1} v_3 v_2 v_{i_2}$ are of length $3$ but that containing $v_{i_2} v_2 v_{i_3}$ is of length larger than $3$, then $i_1 = i_2 = 5$, $i_3 = |V(G)|$ and $\{v_0, v_5\}$ is a separator of $G$. If the faces in $G_{\textrm{ext}}$ containing $v_{i_1} v_3 v_4$, $v_{i_1} v_3 v_2 v_{i_2}$ and $v_{i_2} v_2 v_{i_3}$ are of length $3$, then $i_1 = i_2 = 5$, $i_3 = 6$ and $\{v_0, v_6\}$ is a separator of $G$. In any case it contradicts that $G$ is $3$-connected.
	
	Finally, we consider the case when $i = 3$. It is clear that $d_{G_{\textrm{ext}}}(v_1) = 3$ and $d_{G_{\textrm{ext}}}(v_2) = 4$. Let $v_{i_1}, v_{i_2}$ be the neigbors of $v_2$ other than $v_1, v_3$, and $v_{i_3}$ be the neigbor of $v_1$ other than $v_0, v_2, v_3$, for some $3 < i_1 < i_2 \leq i_3 < |V(G)|$. If the face in $G_{\textrm{ext}}$ containing $v_{i_1} v_2 v_3$ is of length larger than $3$, then $i_1 = |V(G)| - 2$ and $i_2 = i_3 = |V(G)| - 1$, which has been shown to be not possible. If the face in $G_{\textrm{ext}}$ containing $v_{i_1} v_2 v_3$ is of length $3$ but that containing $v_{i_1} v_2 v_{i_2}$ is of length larger than $3$, then $i_1 = 4$, $i_2 = i_3 = |V(G)| - 1$ and $d_{G_{\textrm{int}}}(v_0) = 4$. Let $v_{i_4}$ be the neighbor of $v_0$ other than $v_{-1}, v_1, v_3$ for some $3 < i_4 \leq |V(G)| - 2$. If the faces in $G_{\textrm{int}}$ containing $v_{-1} v_0 v_{i_4}$ is of length larger than $3$, then $i_4 = 4$, which has been shown to be not possible. Hence $v_{-1} v_0 v_{i_4} v_{-1}$ is $2$-triangle, $i_4 = |V(G)| - 2$ and $\{v_{-2}, v_4\}$ is a separator of $G$, which is not possible. If the faces in $G_{\textrm{ext}}$ containing $v_{i_1} v_2 v_3$ and $v_{i_1} v_2 v_{i_2}$ are of length $3$, then $i_1 = 4$ and $i_2 = 5$. Swapping $v_2$ and $v_3$ yields a planar embedding of more $2$-triangles (see Figure~\subref*{subfig:1triangleswap}), which contradicts the maximality of the number of $2$-triangles. Hence we can conclude that there is no $1$-triangle in the plane graph $G$.
	
	It is clear that $G$ is the square of a cycle of length $|V(G)|$ if it has a planar embedding with neither $0$- nor $1$-triangle. To find a cycle of the desired length, one can apply Algorithm~\ref{alg:overloaddischarge} for $T_{\textrm{int}}$ if $|E(G_{\textrm{int}})| > \frac32 |V(G)|$, or if there is no face of length $\frac12 |V(G)|$ in $G_{\textrm{int}}$ and $G_{\textrm{ext}}$, otherwise, do swaps of some vertex pairs at most once for each face to obtain a planar embedding of the square of a cycle of length $|V(G)|$ with neither $0$- nor $1$-triangle, then a cycle of length $\frac12 |V(G)|$ can be easily found in such planar embedding in linear time.
\end{proof}

\begin{figure}[!ht]
	\centering
	\subfloat[Swap $v_0, v_1$, and $v_3, v_4$. \label{subfig:0triangleswap}]{%
		\includegraphics[scale=1.2]{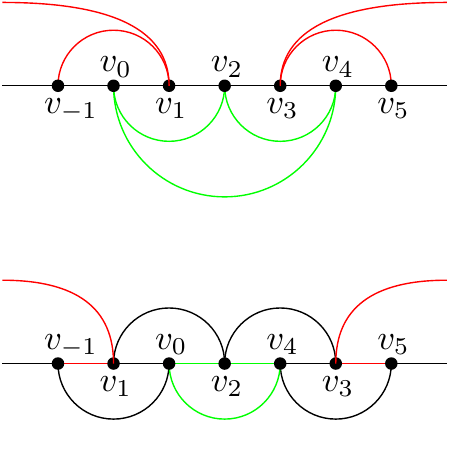}
	}
	\hspace{1.8cm}
	\subfloat[Swap $v_2, v_3$.
	\label{subfig:1triangleswap}]{%
		\includegraphics[scale=1.2]{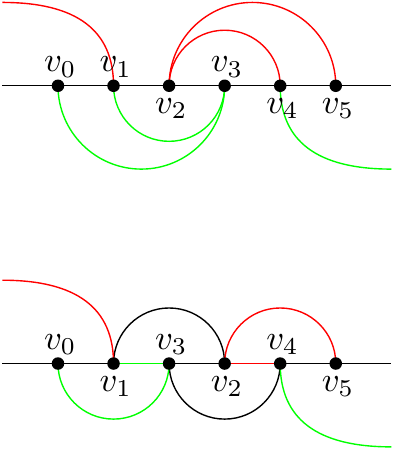}
	}
	\caption{Swap vertices to obtain planar embedding of more $2$-triangles. Edges in $C$, $C_{\textrm{int}}$ and $C_{\textrm{ext}}$ are indicated as black, green and red, respectively.}
	\label{fig:triangleswap}
\end{figure}

\appendix

\section{Find Subsets of Specified Sum via Weighted Subtrees} \label{sec:number}

We here consider two classic $\mathsf{NP}$-complete problems, \textsc{SubsetSum} and \textsc{Partition}, and introduce a new one, \textsc{SubtreeSum}. Let $A := \{a_1, \dots, a_{N}\}$ be a multiset of $N$ positive integers. Given a target value $k \in \mathbb{N}$, $\textsc{SubsetSum}(A, k)$ asks whether there exists $B \subseteq A$ with $\sum B = k$, where $\sum B := \sum_{a \in B} a$. $\textsc{Partition}(A)$ is defined as $\textsc{SubsetSum}(A, \sum A/2)$, assuming $\sum A$ is even. Given a tree $T$ with vertex weights $c: V(T) \rightarrow \mathbb{N}$, and $k \in \mathbb{N}$, $\textsc{SubtreeSum}(T, c, k)$ asks whether there exists a subtree $S$ of $T$ with $c(S) = k$.

Note that \textsc{Partition} is a subproblem of \textsc{SubsetSum}, and \textsc{SubsetSum} can be linear-time reduced to \textsc{SubtreeSum}. To solve $\textsc{SubsetSum}(A, k)$, we can take $T$ to be the \emph{star} with $V(T) = \{v_1, \dots, v_{N}\}$ and assign the vertex weights $c(v_i) := a_i$ for $i = 1, \dots, N$, then $\textsc{SubsetSum}(A, k)$ is true if and only if $\textsc{SubtreeSum}(T, c, k)$ or $\textsc{SubtreeSum}(T, c, \sum A - k)$ is true. In particular, \textsc{SubtreeSum} is $\mathsf{NP}$-complete.

We consider the case that $\sum A \leq 2N - 2$. Set $T$ be a path with $|V(T)| = N =: N_1$ with vertex weights $c$ assigned as above, $N_2 := \sum A$, $g := 1$, $h := 2N - \sum A$. Lemma~\ref{lem:ksubtree} implies the following.

\begin{theorem}
	Let $A := \{a_1, \dots, a_{N}\}$ be a multiset. If $\sum A \leq 2N - 2$ and $\sum A - N + 1 \leq k \leq N$, then $\textsc{SubsetSum}(A, k)$ is true if $a \leq k$ for every $a \in A$. A solution can be found in $O(N)$ time if one exists.
\end{theorem}

\begin{corollary}
	Let $A := \{a_1, \dots, a_{N}\}$ be a multiset with $\sum A$ even. If $N \geq \sum A/2 + 1$, then $\textsc{Partition}(A)$ is true if and only if $a \leq \sum A/2$ for every $a \in A$. A solution can be found in $O(N)$ time if one exists.
\end{corollary}

The criterion $N \geq \sum A/2 + 1$ is tight, since $\textsc{Partition}(A)$ is false for the instance $A = \{2, \dots, 2\}$ with $N = \sum A/2 \geq 2$ odd.

Given an instance $(A, k)$ of \textsc{SubsetSum} with $k \leq \sum A/2$, a standard trick tells us that $\textsc{SubsetSum}(A, k)$ is true if and only if $\textsc{Partition}(A \cup \{\sum A - 2k\})$ is true. It yields the following criterion, which is also tight.

\begin{corollary}
	Let $A := \{a_1, \dots, a_{N}\}$ be a multiset and $k \leq \sum A/2$. If $N \geq \sum A - k$, then $\textsc{SubsetSum}(A, k)$ is true if and only if $a \leq \sum A - k$ for every $a \in A$. A solution can be found in $O(N)$ time if one exists.
\end{corollary}

\section*{Acknowledgments}

The author is very thankful to Samuel Mohr for motivating the problem and sharing his elegant idea; to Tom\'{a}\v{s} Madaras for a comprehensive survey on cycle spectra; and to Matthias Kriesell and Jens M. Schmidt for helpful discussions and comments.

\bibliographystyle{abbrv}
\bibliography{paper}

\begin{thebibliography}{10}

\bibitem{Alon1995}
N.~Alon, R.~Yuster, and U.~Zwick.
\newblock Color-coding.
\newblock {\em J. ACM}, 42(4):844--856, 1995.

\bibitem{Alon1997}
N.~Alon, R.~Yuster, and U.~Zwick.
\newblock Finding and counting given length cycles.
\newblock {\em Algorithmica}, 17(3):209–223, 1997.

\bibitem{Bondy1975}
J.~A. Bondy.
\newblock Pancyclic graphs: Recent results.
\newblock In A.~Hajnal, R.~Rado, and V.~S\'{o}s, editors, {\em Infinite and
  finite sets (Colloq. Keszthely, 1973; Dedicated to P. Erd\H{o}s on his 60th
  birthday)}, volume~10 of {\em Colloq. Math. Soc. J\'{a}nos Bolyai}, pages
  181--187. North-Holland, Amsterdam, 1975.

\bibitem{Chen2004}
G.~Chen, G.~Fan, and X.~Yu.
\newblock Cycles in 4-connected planar graphs.
\newblock {\em European J. Combin.}, 25:763–780, 2004.

\bibitem{Chiba1989}
N.~Chiba and T.~Nishizeki.
\newblock The hamiltonian cycle problem is linear-time solvable for 4-connected
  planar graphs.
\newblock {\em J. Algorithms}, 10(2):187--211, 1989.

\bibitem{Chiba1985}
N.~Chiba, T.~Nishizeki, S.~Abe, and T.~Ozawa.
\newblock A linear algorithm for embedding planar graphs using {PQ}-trees.
\newblock {\em J. Comput. System Sci.}, 30(1):54--76, 1985.

\bibitem{Choudum1977}
S.~A. Choudum.
\newblock Some 4-valent, 3-connected, planar, almost pancyclic graphs.
\newblock {\em Discrete Math.}, 18(2):125--129, 1977.

\bibitem{Cui2009}
Q.~Cui, Y.~Hu, and J.~Wang.
\newblock Long cycles in 4-connected planar graphs.
\newblock {\em Discrete Math.}, 309(5):1051--1059, 2009.

\bibitem{Feder2010}
T.~Feder and R.~Motwani.
\newblock Finding large cycles in {H}amiltonian graphs.
\newblock {\em Discrete Appl. Math.}, 158(8):882--893, 2010.

\bibitem{Fijavz2002}
G.~Fijav\v{z}, M.~Juvan, B.~Mohar, and R.~\v{S}krekovski.
\newblock Planar graphs without cycles of specific lengths.
\newblock {\em European J. Combin.}, 23(4):377--388, 2002.

\bibitem{Hierholzer1873}
C.~Hierholzer.
\newblock Ueber die {M}\"{o}glichkeit, einen {L}inienzug ohne {W}iederholung
  und ohne {U}nterbrechung zu umfahren.
\newblock {\em Math. Ann.}, 6(1):30–32, 1873.

\bibitem{Hornak2008}
M.~Hor\v{n}\'{a}k and Z.~Kockov\'{a}.
\newblock On planar graphs arbitrarily decomposable into closed trails.
\newblock {\em Graphs Combin.}, 24(1):19--28, 2008.

\bibitem{Madaras}
T.~Madaras and M.~Tam\'{a}\v{s}ov\'{a}.
\newblock Minimal unavoidable sets of cycles in plane graphs with restricted
  minimum degree and edge weight.
\newblock Manuscript.

\bibitem{Malkevitch1971}
J.~Malkevitch.
\newblock On the lengths of cycles in planar graphs.
\newblock In M.~Capobianco, J.~Frechen, and M.~Krolik, editors, {\em Recent
  Trends in Graph Theory (Proc. Conf., New York 1970)}, volume 186 of {\em
  Lecture Notes in Mathematics}, page 191–195. Springer, Berlin, 1971.

\bibitem{Malkevitch1978}
J.~Malkevitch.
\newblock Cycle lengths in polytopal graphs.
\newblock In Y.~Alavi and D.~Lick, editors, {\em Theory and Applications of
  Graphs (Proc. Conf., Michigan 1976)}, volume 642 of {\em Lecture Notes in
  Computer Science}, page 364–370. Springer, Berlin, 1978.

\bibitem{Malkevitch1988}
J.~Malkevitch.
\newblock Polytopal graphs.
\newblock In L.~Beineke and R.~Wilson, editors, {\em Selected Topics in Graph
  Theory}, volume~3, page 169–188. Academic Press, 1988.

\bibitem{Mohr2018}
S.~Mohr.
\newblock Personal communication, 2018.

\bibitem{Plummer1975}
M.~D. Plummer.
\newblock Problems.
\newblock In A.~Hajnal, R.~Rado, and V.~S\'{o}s, editors, {\em Infinite and
  finite sets (Colloq. Keszthely, 1973; Dedicated to P. Erd\H{o}s on his 60th
  birthday)}, volume~10 of {\em Colloq. Math. Soc. J\'{a}nos Bolyai}, pages
  1549--1550. North-Holland, Amsterdam, 1975.

\bibitem{Sanders1996}
D.~P. Sanders.
\newblock On {H}amilton cycles in certain planar graphs.
\newblock {\em J. Graph Theory}, 21(1):43--50, 1996.

\bibitem{Tarjan1984}
R.~E. Tarjan and U.~Vishkin.
\newblock Finding biconnected componemts and computing tree functions in
  logarithmic parallel time.
\newblock In {\em Proceedings of the 25th Annual Symposium on Foundations of
  Computer Science (FOCS'84)}, pages 12--20, Florida, 1984.

\bibitem{Thomas1994}
R.~Thomas and X.~Yu.
\newblock 4-connected projective-planar graphs are hamiltonian.
\newblock {\em J. Combin. Theory Ser. B}, 62:114–132, 1994.

\bibitem{Thomassen1983}
C.~Thomassen.
\newblock A theorem on paths in planar graphs.
\newblock {\em J. Graph Theory}, 7(2):169--176, 1983.

\bibitem{Trenkler1989}
M.~Trenkler.
\newblock On 4-connected, planar 4-almost pancyclic graphs.
\newblock {\em Math. Slovaca}, 39(1):13--20, 1989.

\bibitem{Tutte1956}
W.~T. Tutte.
\newblock A theorem on planar graphs.
\newblock {\em Trans. Amer. Math. Soc.}, 82:99–116, 1956.

\bibitem{Wang2002}
W.~Wang and K.-W. Lih.
\newblock Choosability and edge choosability of planar graphs without 5-cycles.
\newblock {\em Appl. Math. Lett.}, 15(5):561--565, 2002.

\end{thebibliography}

\end{document}